\documentclass[11pt]{amsart}

\usepackage{amsmath}
\usepackage{amssymb}
\usepackage{latexsym}
\usepackage{amscd}
\usepackage{citesort}
\usepackage{mathrsfs}
\usepackage[all]{xy}

\newdimen\AAdi%
\newbox\AAbo%
\def\AAk#1#2{\s_etbox\AAbo=\hbox{#2}\AAdi=\wd\AAbo\kern#1\AAdi{}}%
\def\AAr#1#2#3{\s_etbox\AAbo=\hbox{#2}\AAdi=\ht\AAbo\raise#1\AAdi\hbox{#3}}%
\font\tenmsb=msbm10 at 12pt \font\sevenmsb=msbm7 at 8pt
\font\fivemsb=msbm5 at 6pt
\newfam\msbfam
\textfont\msbfam=\tenmsb \scriptfont\msbfam=\sevenmsb
\scriptscriptfont\msbfam=\fivemsb

\newtheorem{theorem}{Theorem}

\newtheorem{remark}[theorem]{Remark}

\newtheorem{proposition}[theorem]{Proposition}

\numberwithin{equation}{section} \numberwithin{theorem}{section}

\renewcommand{\topmargin}{0cm}
\renewcommand{\oddsidemargin}{5mm}
\renewcommand{\evensidemargin}{5mm}
\renewcommand{\textwidth}{150mm}
\renewcommand{\textheight}{230mm}

\def\R{\mathbb R}

\def\na{\nabla}

\def\grs#1#2{\bold G_{#1,#2}}

\def\dt#1{\frac {d\,#1}{d\,t}}

\def\a{\alpha}

\def\p#1{\partial #1}

\def\De{\Delta}
\def\e{\eta}
\def\ep{\varepsilon}
\def\eps{\epsilon}

\def\g{\gamma}

\def\lan{\langle}
\def\ran{\rangle}

\begin{document}

\title[curvature estimates of  ancient solutions]
{curvature estimates of ancient solutions to the mean curvature flow of  higher codimension with convex Gauss image}

\author{Hongbing Qiu}

\address{School of Mathematics and Statistics, Wuhan University, Wuhan 430072,
China}
 \email{hbqiu@whu.edu cn}

\author{Y.L. Xin}
\address{Institute of Mathematics, Fudan University, Shanghai 200433, China}
\email{ylxin@fudan.edu.cn}

 \thanks{
 The first author is partially supported by NSFC (No. 11771339) and Hubei Provincial Natural Science Foundation of China (No. 2021CFB400).  }

\begin{abstract} By carrying out  refined curvature estimates, we prove better rigidity theorems
 of complete noncompact ancient solutions to the mean curvature flow in higher codimension under various Gauss image restriction.

\vskip12pt



\end{abstract}


 \maketitle

\section{Introduction}

Let $F: M^{n} \rightarrow \mathbb{R}^{m+n}$ be an isometric
immersion from an $n$-dimensional oriented Riemannian manifold $M$
to the Euclidean space $\mathbb{R}^{m+n}$. The mean curvature flow
(MCF) in Euclidean space is a one-parameter family of immersions
$F_t= F(\cdot, t): M^n \rightarrow \mathbb{R}^{m+n}$ with the
corresponding images $M_t=F_t(M)$ such that
\begin{equation}\label{eqn-MCF}
\begin{cases}\aligned
\frac{\partial}{\partial t}F(x,t)=& H(x,t), x\in M,\\
F(x,0)=&F(x),
\endaligned
\end{cases}
\end{equation}
is satisfied, where $H(x, t)$ is the mean curvature vector of $M_t$
at $F(x, t)$ in $\mathbb{R}^{m+n}$.

Singularity issue of the MCF is important. Besides self-shrinkers and translating solitons, the  ancient solutions also relate to the issue. By definition, an ancient solution to the mean curvature flow is a solution
 which is defined on a time interval of the form $(-\infty, T_0)$ with $T_0 < \infty$.
 These solutions typically arise as tangent flows near singularities  and model the
  asymptotic profile of the mean curvature flow near a singularity (see \cite{HuiSin199}\cite{HuiSin299}).
   Since the diffusion has had an infinite amount of time to take effect, ancient solutions can be
   expected to exhibit rigidity phenomena. There have been plenty of works for compact ancient solutions
    of the mean curvature flow\cite{ADS19}\cite{ADS20}\cite{BIS17}\cite{BL16}\cite{BLT21}\cite{CM19}\cite{CM22}\cite{DHS10}\cite{HH16}\cite{HS15}\cite{Lan17}\cite{Lan17a}\cite{LL20}\cite{LN21}\cite{LXZ21}\cite{SS19}\cite{Wan11a}\cite{Whi03}.

For the noncompact ancient solutions, Brendle-Choi \cite{BC19}
showed that the rotationally symmetric bowl soliton is the only
noncompact ancient solution to the mean curvature flow in
$\mathbb{R}^3$ which is strictly convex and non-collapsed.
Afterward, they \cite{BC21} generalized this result to higher
dimensions under an additional assumption of uniform two-convexity.
Recently, Haslhofer and coauthors \cite{CH21}\cite{DH21}\cite{DH21b}
gave the classification of the ancient noncollapsed flows.  In
\cite{Kun16}, the author showed that there are no nontrivial ancient
solutions with bounded slope and bounded mean curvature in
codimension one. This result was generalized  to arbitrary
codimension by Guan-Xu-Zhao \cite{GXZ23}, and  they obtained that if
the $v$-function satisfies $v\leq v_1 < \sqrt{2}$ and the mean
curvature is bounded, then any ancient solution in higher
codimension has to be an affine subspace. The precise
definition of the slope function $v$ will be given in section 2. See also the related
work by the first author \cite{Qiu22}.

In this paper, we study the noncompact ancient solutions to the mean
curvature flow in higher codimension. Some kind of convexity
conditions for the codimension one case are essential for all
previous work mentioned above. Nevertheless, convexity condition is
not easy to generalized to higher codimension. On the other hand, it
is natural to consider some kind of convex Gauss image situation.
For a submanifold in the Euclidean space the target manifold of the
Gauss map is the Grassmannian manifold. We already have the largest
geodesic convex subset  $B_{JX}(P_0)$ of the Grassmannian
manifold \cite{JX}. With the aid of analytic method
 we can carry out  point-wise curvature estimates, as shown in (\ref{es1}) for ancient solution to the mean curvature flow under the Gauss image
contained in a compact subset of $B_{JX}(P_0)$.  This leads to a
rigidity result of ancient solutions. The Gauss image restriction
can also be described by upper bound of the slope function $v$. We
know that  the domain of $v < 2$ in the Grasmannian manifold is
contained in $B_{JX}(P_0)$ \cite{JX}. In fact, we can do better.
Namely, under the slope function  satisfying $v\le v_0<3$ we can
also carry out curvature estimates and obtain corresponding rigidity
results, see Theorem \ref{thm5}, Theorem \ref{thm-anc6}  and Theorem
\ref{thm-anc7}.

\bigskip

\section{Preliminaries}
\medskip

 The Grassmannian manifold $\grs{n}{m}$ can be viewed as a submanifold of some Euclidean space via the Pl${\ddot{\rm u}}$cker embedding.
  The restriction of the Euclidean inner product on $\grs{n}{m}$ is denoted by $w: \grs{n}{m} \times \grs{n}{m} \to \mathbb{R}$
\[
w(P, Q) = \lan e_1 \wedge\cdot\cdot\cdot \wedge e_n, f_1\wedge
\cdot\cdot\cdot \wedge f_n \ran = \det W,
\]
where $\{ e_1, \cdot\cdot\cdot, e_n \}$ is an oriented orthonormal
basis of $P$, $\{ f_1, \cdot\cdot \cdot, f_n\}$ is an oriented
orthonormal basis of $Q$ and $W=(\lan e_i, f_j \ran)$.  It is well
known that
\[
W^TW=O^T\Lambda O
\]
with $O$ an orthogonal matrix and $\Lambda = {\rm diag}(\mu_1^{2},
\cdot\cdot\cdot, \mu_n^{2})$. Here each $0\leq \mu_i^{2}\leq 1$.
Putting $p:= \min\{ m, n \}$, then at most $p$ elements in $\{
\mu_1^{2}, \cdot\cdot\cdot, \mu_n^{2} \}$ are not equal to 1.
Without loss of generality, we can assume $\mu_i^{2} = 1$ whenever
$i>p$. We also note that the $\mu_i^{2}$ can be expressed as
\[
\mu_i^{2} = \frac{1}{1+\lambda_i^{2}}
\]
with $\lambda_i \in [0, +\infty)$.

The Jordan angles between $P$ and $Q$ are defined by
\[
\theta_i = \arccos (\mu_i), \quad 1\leq i \leq p.
\]
The distance between $P$ and $Q$ is defined by
\[
d(P, Q) = \sqrt{\sum \theta_i^{2}}.
\]
It is a natural generalization of the canonical distance of
Euclidean sphere. Thus we have
\[
\lambda_i = \tan\theta_i.
\]
In the sequel, we shall assume $m\geq n$ without loss of generality.

Now we fix $P_0 \in \grs{n}{m}$. We represent it by the $n$-vector
$\ep_1\wedge\cdot\cdot\cdot \wedge \ep_i \wedge \cdot\cdot\cdot
\ep_n$. We choose $m$ vectors $\ep_{n+\a}$, such that $\{ \ep_i,
\ep_{n+\a} \}$ form an orthonormal basis of $\mathbb{R}^{m+n}$.
Denote
\[
\mathbb{U}:= \{ P \in \grs{n}{m}: w(P, P_0) > 0 \}.
\]
The $v$-function will be
\[
v(\cdot, P_0) := w^{-1}(\cdot, P_0) \quad {\rm on} \quad \mathbb{U}.
\]
For arbitrary $P \in \mathbb{U}$, it is easy to see that
\[
v(P, P_0) =  \prod_{\a=1}^n \sec \theta_\a = \prod_{\a=1}^n
\frac{1}{\mu_\a},
\]
where $\theta_1, \cdot\cdot\cdot, \theta_n$ denotes the Jordan
angles between $P$ and $P_0$.

For convenience, we define a subset in $\mathbb{U}$ by
\[
\mathbb{U}_2:=\{ P\in \mathbb{U}: v(P, P_0)< 2 \}.
\]

In \cite{JX}, Jost and the second author found the largest geodesic
convex subset $B_{JX}(P_0)$, which is defined by sum of any two
Jordan angles being less than $\frac{\pi}{2}$ for any point $P\in
B_{JX}(P_0)$. It is easy to see that
\begin{equation}\label{eqn-LGCS}
\mathbb{U}_2\subset B_{JX}(P_0)=\{ P\in \mathbb U:\lambda_i\lambda_j
< 1 \quad {\rm for} \quad {\rm every} \quad i\neq j \}
\end{equation}

If $M^n$ is an oriented submanifold in $\mathbb{R}^{m+n}$, we can
define the Gauss map $\gamma: M \to \grs{n}{m}$ that is obtained by
parallel translation of $T_pM$ to the origin in the ambient space
$\mathbb{R}^{m+n}$. Then we can define natural functions on $M^n$
from functions on $\grs{n}{m}$ via the Gauss map:
\[
w:=w(\cdot, P_0)\circ \gamma, \qquad v:=v(\cdot, P_0)\circ \gamma.
\]
In particular, $v$ is called the slope function.
\bigskip

Let $\na$ and $\overline{\na}$ be the Levi-Civita connections on
$M^n$ and $\R^{m+n}$, respectively. The second fundamental form $B$
of $M^{n}$ in $\mathbb{R}^{m+n}$ is defined by
\[
B_{UW}:= (\overline{\na}_U W)^N
\]
for $U, W \in \Gamma(TM^n)$. We use the notation $( \cdot )^T$ and
$( \cdot )^N$ for the orthogonal  projections into the tangent
bundle $TM^n$ and the normal bundle $NM^n$, respectively. For $\nu
\in \Gamma(NM^n)$ we define the shape operator $A^\nu: TM^n
\rightarrow TM^n$ by
\[
A^\nu (U):= - (\overline{\na}_U \nu)^T
\]
Taking the trace of $B$ gives the mean curvature vector $H$ of $M^n$
in $\mathbb{R}^{m+n}$ and
\[
H:= \hbox{trace} (B) = \sum_{i=1}^{n} B_{e_ie_i},
\]
where $\{ e_i \}$ is a local orthonormal frame field of $M^n$.

\bigskip

Let $B_R(o)$ be a Euclidean closed ball of radius $R$ centered at
the origin $o \in \mathbb{R}^{m+n}$ and $B_{R, T}(o):= B_R(o) \times
[-T, 0]  \subset \mathbb{R}^{m+n}\times (-\infty, +\infty)$. We may
consider $\mathcal{M}_T$ as the space-time domain
\[
\{ (F(p, t), t): p\in M^n, t\in [-T, 0] \} \subset \mathbb{R}^{m+n}
\times (-\infty, \infty).
\]
Let $D_{R, T}(o) = \{ (x, t) \in M^n \times [-T, 0]: F(x, t) \in
B_R(o) \}.$

\bigskip

\section{Evolution equations}

\medskip

By Lemma 3.1 in \cite{Xin08}, we have
\begin{equation}\label{eqn-VLap-var22bb}\aligned
\left( \De - \p_t \right) |B|^2\geq 2|\na |B||^2 - 3|B|^4.
\endaligned
\end{equation}

We have the following equation which is the parabolic version of the well-known Ruh-Villms' theorem.
\begin{proposition}(see \cite{Wan03})
\begin{equation}
\dt{\g}=\tau(\g(t)).\label{prv}
\end{equation}
\end{proposition}

Let $\hat h$ be any function on $\grs{n}{m}$. The composition
function $\hat h\circ \g$ of $\hat h$ with the Gauss map $\g$
defines a function on $M_t=F(M, t)$. By (\ref{prv}) we have
$$\dt{}(\hat h\circ\g)=d\hat h\left(\dt{\g}\right)=d\hat h(\tau(\g)).$$
By the composition formula (see \cite{X2}, p.28)
$$\De(\hat h\circ\g)=\text{Hess}(\hat h)(\g_*e_i,\g_*e_i)+d\hat h(\tau(\g)),$$
where $\{e_i\}$ is a local orthonormal frame field on $M_t$. It
follows that
\begin{equation}
\left(\dt{}-\De\right)\hat h\circ\g=-\text{Hess}(\hat
h)(\g_*e_i,\g_*e_i).
\end{equation}

If the slope function $v\le v_0<3,$ from the formulas (2.8), (3.13),
(3.16), Lemma 3.1 and Lemma 3.2 in \cite{JXY}, we have
\begin{equation}\label{eqn-h-c}\aligned
\sum_{i=1}^n{\rm Hess}(v(\cdot, P_0))(d\gamma(e_i),
d\gamma(e_i))\circ \gamma \geq& \ep_0 |B|^2,
\endaligned
\end{equation}
where $\ep_0$ is a positive constant depending only on $v_0$. It follows that
\begin{equation}\label{ev1}
\left(\dt{}-\De\right)v =-\sum_{i=1}^n\text{Hess}(v(\cdot,
P_0))(d\g(e_i),d\g(e_i)\circ \g \le -\ep_0 |B|^2.
\end{equation}
Using maximum principle we conclude that $v\le v_0<3$ remains true
provided it is valid in the initial submanfold.

Let $E_{i\alpha}$ be the matrix with 1 in the intersection of row
$i$ and column $\alpha$ and $0$ otherwise. Then $\sec
\theta_i\sec\theta_\alpha E_{i\alpha}$ form an orthonormal basis of
$T_P \grs{n}{m}$ with respect to the canonical Riemannian metric $g$
on $\grs{n}{m}$. Denote its dual frame by $\omega_{i\alpha}$. Then
$g$ can be written as
\[
g=\sum_{i, \alpha} \omega_{i\alpha}^2.
\]
Denote $\omega_{\beta\alpha}=0$ for $\beta\geq n+1$.

From the formula (3.8) in \cite{XY08}, we get
\begin{equation}\label{eqn-BJ1}\aligned
dv(\cdot, P_0)=\sum_{1\leq j \leq n}\lambda_j v(\cdot,
P_0)\omega_{jj},
\endaligned
\end{equation}
and
\begin{equation}\label{eqn-BJ2}\aligned
v(\cdot, P_0)^{-1}{\rm Hess}(v(\cdot, P_0))=g+
\sum_{\alpha,\beta}\lambda_{\alpha}\lambda_{\beta}(\omega_{\alpha\alpha}\otimes
\omega_{\beta\beta}+\omega_{\alpha\beta}\otimes
\omega_{\beta\beta}).
\endaligned
\end{equation}
By the equalities (\ref{eqn-BJ1}) and (\ref{eqn-BJ2}), we obtain
(see also (3.9) in \cite{DXY} or (2.12) in \cite{DJX21})
\begin{equation}\label{eqn-BJ3}\aligned
{\rm Hess}\log v(\cdot, P_0)=g+ \sum_{1\leq j\leq
n}\lambda_{j}^2\omega_{jj}^2 + \sum_{1\leq i, j\leq n, i\neq
j}\lambda_i\lambda_j\omega_{ij}\otimes \omega_{ji}.
\endaligned
\end{equation}
Since
\begin{equation}\label{eqn-BJ4}\aligned
\omega_{jk}(d\gamma(e_i))= \omega_{jk}(h_{\alpha,il}e_{\alpha l})=
h_{k, ij}.
\endaligned
\end{equation}
Therefore combining (\ref{eqn-BJ3}) with (\ref{eqn-BJ4}), it follows
that
\begin{equation}\label{eqn-BJ5}\aligned
&\sum_{i}{\rm Hess}(\log v(\cdot, P_0))(d\gamma(e_i),
d\gamma(e_i))\circ \gamma \\
=& |B|^2 + \sum_{i, j}\lambda_j^{2}h_{j,ij}^2 + \sum_{i, j\neq
k}\lambda_j\lambda_k
h_{k,ij}h_{j,ik}\\
=& \sum_{\alpha>n,ij}h_{\alpha,ij}^2 + \sum_{i}h_{i,ii}^2 +
\sum_{i\neq j}h_{j,ii}^2 +2\sum_{i\neq j}h_{i,ij}^2 \\
& +\sum_{i,j,k \quad{\rm mutually \quad distinct}}h_{k,ij}^2 +
\sum_{i}\lambda_{i}^2h_{i,ii}^2 + \sum_{i\neq
j}\lambda_i^{2}h_{i,ij}^2 \\
&+ 2\sum_{i\neq j}\lambda_i\lambda_jh_{i,ji}h_{j,ii}+ \sum_{i,j,k
\quad{\rm
mutually \quad distinct}}\lambda_i\lambda_jh_{i,jk}h_{j,ik}\\
=& \sum_{\alpha>n,ij}h_{\alpha,ij}^2 +
\sum_{i}(1+\lambda_i^{2})h_{i,ii}^2 + \sum_{i\neq j}h_{j,ii}^2
+\sum_{i\neq j}(2+\lambda_i^{2})h_{i,ij}^2 \\
&    + 2\sum_{i\neq j}\lambda_i\lambda_jh_{i,ji}h_{j,ii}
+\sum_{i,j,k \quad{\rm mutually \quad
distinct}}\left(h_{k,ij}^2+\lambda_i\lambda_jh_{i,jk}h_{j,ik}\right).
\endaligned
\end{equation}
Direct computation gives us
\begin{equation}\label{eqn-BJ6}\aligned
&\sum_{i,j,k \quad{\rm mutually \quad
distinct}}\left(h_{k,ij}^2+\lambda_i\lambda_jh_{i,jk}h_{j,ik}\right)\\
=& 2\sum_{i<j<k} \left(
h_{k,ij}^2+h_{j,ki}^2+h_{i,jk}^2+\lambda_i\lambda_jh_{i,jk}h_{j,ik}\right.\\
&\left.+\lambda_i\lambda_kh_{i,kj}h_{k,ij}+\lambda_j \lambda_k h_{j,
ki}h_{k,ji}\right).
\endaligned
\end{equation}
If $\sup_{i\neq j} \lambda_i \lambda_j \leq \lambda_0<1$
($\lambda_0$ is a constant) , then by the Cauchy inequality, we
obtain
\begin{equation}\label{eqn-BJ66}\aligned
&\sum_{i,j,k \quad{\rm mutually \quad
distinct}}\left(h_{k,ij}^2+\lambda_i\lambda_jh_{i,jk}h_{j,ik}\right)\\
\geq & 2(1-\lambda_0)\sum_{i<j<k} \left(
h_{k,ij}^2+h_{j,ki}^2+h_{i,jk}^2\right)\\
=& (1-\lambda_0)\sum_{i,j,k \quad{\rm mutually \quad
distinct}}h_{k,ij}^2.
\endaligned
\end{equation}
 The Cauchy inequality implies that
\begin{equation}\label{eqn-BJ11}\aligned
2\lambda_i\lambda_j h_{i,ji}h_{j,ii} \leq 2\lambda_0
h_{i,ji}h_{j,ii} \leq 2\lambda_0h_{i,ji}^2 +
\frac{\lambda_0}{2}h_{j,ii}^2.
\endaligned
\end{equation}
Substituting (\ref{eqn-BJ66}) and (\ref{eqn-BJ11}) into
(\ref{eqn-BJ5}), we get
\begin{equation}\label{eqn-BJ12}\aligned
&\sum_{i}{\rm Hess}(\log v(\cdot, P_0))(d\gamma(e_i),
d\gamma(e_i))\circ \gamma \\
\geq & \sum_{\alpha>n,ij}h_{\alpha,ij}^2 +
\sum_{i}(1+\lambda_i^{2})h_{i,ii}^2 + \sum_{i\neq j}h_{j,ii}^2
+\sum_{i\neq j}(2+\lambda_i^{2})h_{i,ij}^2 \\
&    - \sum_{i\neq j}\left(2\lambda_0 h_{i,ji}^2 +
\frac{\lambda_0}{2} h_{j,ii}^2\right) + (1-\lambda_0)\sum_{i,j,k
\quad{\rm mutually \quad distinct}}h_{k,ij}^2.
\endaligned
\end{equation}
It follows that
\begin{equation}\label{eqn-BJ13}\aligned
&\sum_{i}{\rm Hess}(\log v(\cdot, P_0))(d\gamma(e_i),
d\gamma(e_i))\circ \gamma \\
\geq & \sum_{\alpha>n,ij}h_{\alpha,ij}^2 + \sum_i h_{i,ii}^2 +
\sum_{i, j}\lambda_i^{2}h_{i,ij}^2 + 2(1-\lambda_0)\sum_{i\neq
j}h_{i,ij}^2 \\
&+ (1-\lambda_0)\sum_{i\neq j} h_{j,ii}^2 + (1-\lambda_0)\sum_{i,j,k
\quad{\rm mutually \quad distinct}}h_{k,ij}^2.
\endaligned
\end{equation}
Then we have
\begin{equation}\label{eqn-BJ14}\aligned
\sum_{i}{\rm Hess}(\log v(\cdot, P_0))(d\gamma(e_i),
d\gamma(e_i))\circ \gamma \geq & (1-\lambda_0)|B|^2 +
\sum_{i,j}\lambda_i^{2}h_{i,ij}^2 \\
\geq & (1-\lambda_0)|B|^2.
\endaligned
\end{equation}

It follows that if $\sup_{i\neq j} \lambda_i \lambda_j \leq
\lambda_0<1$ then
\begin{equation}\label{ev2}
\left(\dt{}-\De\right)v=-\sum_{i=1}^n\text{Hess}(v(\cdot,
P_0))(d\g(e_i),d\g(e_i)\circ \g \le -(1-\lambda_0) |B|^2.
\end{equation}
It was proved in \cite{TW} that the the condition $\sup_{i\neq j} \lambda_i \lambda_j \leq \lambda_0<1$ is remained true under the mean
curvature flow  provided it is valid in the initial submanifold.

\begin{remark}\label{eve}

Under the weaker condition $\sup_{i\neq j} \lambda_i \lambda_j \leq \lambda_0<\sqrt{2}$ similar inequality as (\ref{eqn-BJ14}) is   still valid
(see \cite{DJX21}), but it is unknown whether it remains true under the mean curvture flow.

\end{remark}

\section{Curvature estimates}

\medskip

In this section we carry out
local curvature estimates  and obtain corresponding rigidity
results of complete noncompact ancient solutions in higher
codimension under Gauss image restrictions.

\begin{theorem}\label{thm5}

Let $F: M^n\times [-T, 0] \to \mathbb{R}^{m+n}$ be a solution to the
mean curvature flow ($m\geq 2$). If the slope function $v$ satisfies
\[
 v \leq v_0 < 3
\]
on the initial submanifold, where $v_0$ is a constant.
 Then there exists a positive constant $C$ which is independent of $R$ and $T$, such that
\begin{equation}\label{eqn-B1v3}
\sup_{D_{\frac{R}{2}, \frac{T}{2}}(o)} |B |\leq C \left( \frac{1}{R}
+ \frac{1}{\sqrt{T}} \right).
\end{equation}

\end{theorem}

\begin{proof}
Denote $h:= kv$, where $k$ is a positive constant which will be
determined later. From (\ref{ev1}), we derive
\begin{equation}\label{eqn-h-d}\aligned
\left( \De - \p_t \right) h \geq k\ep_0|B|^2.
\endaligned
\end{equation}
It follows that
\begin{equation}\label{eqn-h-d1}\aligned
\left( \De - \p_t \right) e^h = & e^h |\na h|^2 + e^h \left( \De - \p_t \right) h  \\
  \geq & e^h |\na h|^2 + e^h  k\ep_0|B|^2.
\endaligned
\end{equation}
Recall (\ref{eqn-VLap-var22bb}), we have
$$\aligned
\left( \De - \p_t \right) |B|^2\geq 2|\na |B||^2 - 3|B|^4.
\endaligned
$$
From (\ref{eqn-h-d1}) and (\ref{eqn-VLap-var22bb}), we obtain
\begin{equation}\label{eqn-h-d2}\aligned
\left( \De - \p_t \right) (|B|^2e^h) = & e^h(\De-\p_t)|B|^2 + 2\lan \na|B|^2, \na e^h \ran + |B|^2(\De-\p_t)e^h \\
\geq & e^h \left( 2|\na |B||^2 - 3|B|^4 \right) + 2\lan \na|B|^2, \na e^h \ran \\
& + |B|^2\left(  e^h |\na h|^2 + e^h k\ep_0|B|^2 \right) \\
=& e^h \left[ \left( k\ep_0-3 \right)|B|^4 + 2|\na |B||^2 + |B|^2|\na h|^2 \right] \\
&+ 2\lan \na|B|^2, \na e^h \ran.
\endaligned
\end{equation}
Since
\begin{equation}\label{eqn-h-d3}\aligned
2\lan \na|B|^2, \na e^h \ran =& e^{-h} \lan \na(|B|^2e^h), \na e^h \ran - |B|^2e^h|\na h|^2 + 2|B|e^h \lan\na|B|, \na h \ran \\
\geq& e^{-h} \lan \na(|B|^2e^h), \na e^h \ran - |B|^2e^h|\na h|^2 \\
&-2|\na |B||^2e^h - \frac{1}{2} |B|^2 |\na h|^2e^h \\
=& e^{-h} \lan \na(|B|^2e^h), \na e^h \ran -2|\na |B||^2e^h -
\frac{3}{2} |B|^2 |\na h|^2e^h.
\endaligned
\end{equation}
Therefore from  (\ref{eqn-h-d2}) and (\ref{eqn-h-d3}), we get
\begin{equation}\label{eqn-h-d4}\aligned
\left( \De - \p_t \right) (|B|^2e^h) \geq & e^h \left[ \left( k\ep_0 -3 \right)|B|^4 - \frac{1}{2} |B|^2 |\na h|^2 \right]  \\
&+  e^{-h} \lan \na(|B|^2e^h), \na e^h \ran.
\endaligned
\end{equation}

Define the function on $\mathcal{M}_T$ by
\[
 f = |B|^2 e^h.
\]
Then we have
\begin{equation}\label{eqn-h-d5}\aligned
\left( \De - \p_t \right) f  \geq &  \left( k\ep_0 -3 \right)|B|^4
e^h - \frac{1}{2}  f  |\na h|^2 + \lan \na  f , \na h \ran.
\endaligned
\end{equation}
By the assumption that $v <3$, thus
\[
e^h > e^{-3k}\cdot e^{2h}.
\]
It follows that
\begin{equation}\label{eqn-h-d6}\aligned
\left( \De - \p_t \right) f  \geq &  \left( k\ep_0 -3 \right)e^{-3k}
 f^2 - \frac{1}{2}  f  |\na h|^2 + \lan \na
 f , \na h \ran.
\endaligned
\end{equation}

Let $\eta(r, t): \mathbb{R} \times \mathbb{R} \to \mathbb{R}$ be a
smooth function supported on $[-R, R]\times [-T, 0]$, satisfying the
following properties:
\begin{itemize}
    \item[(1)] $\eta(r, t) \equiv 1$ on $[-\frac{R}{2}, \frac{R}{2}] \times [-\frac{T}{2}, 0]$ and $0\leq \eta \leq 1.$
    \item[(2)] $\eta(r, t)$ is decreasing if $r\geq 0$, i.e., $\p_r \eta \leq 0$.
    \item[(3)] $\frac{|\p_r \eta|}{\eta^a}\leq \frac{C_a}{R}, \frac{|\p_r^{2}\eta|}{\eta^a} \leq \frac{C_a}{R^2}$ for $a=\frac{1}{2}, \frac{3}{4}$.
    \item[(4)] $\frac{|\p_t \eta|}{\eta^{\frac{1}{2}}} \leq \frac{C}{T}$.
\end{itemize}
Such a function was explicitly constructed in \cite{Kun16} (see also \cite{LY}\cite{SZ}). 
Let $\phi:= \eta(r(F), t)$, where $r(F):= |F|$.

Let $L:=-\na h$.
Then we derive
 \begin{equation}\label{eqn-h-d7}\aligned
&\left(\De- \p_t \right)(\phi  f) + \lan L, \na(\phi  f) \ran -2 \left\lan \frac{\na \phi}{\phi}, \na(\phi  f) \right\ran \\
=&  f(\De-\p_t)\phi + \phi (\De-\p_t) f + 2\lan \na \phi, \na f \ran \\
&+ \phi\lan L, \na f \ran + f\lan L, \na \phi \ran - \frac{2|\na \phi|^2 f}{\phi} - 2\lan \na \phi, \na f \ran \\
=& f(\De-\p_t)\phi + \phi (\De-\p_t) f -\phi\lan \na h, \na f \ran -
 f \lan \na h, \na \phi \ran - \frac{2|\na
\phi|^2 f}{\phi}.
\endaligned
\end{equation}
Combining (\ref{eqn-h-d6}) with (\ref{eqn-h-d7}), it follows
\begin{equation}\label{eqn-h-d8}\aligned
&\left(\De- \p_t \right)(\phi  f) + \lan L, \na(\phi  f) \ran -2 \left\lan \frac{\na \phi}{\phi}, \na(\phi  f) \right\ran \\
\geq & \left( k\ep_0 -3 \right)e^{-3k} \phi f^2 +  f(\De-\p_t)\phi \\
&  - \frac{1}{2} \phi f  |\na h|^2 -  f \lan \na \phi , \na h \ran -
\frac{2|\na \phi|^2 f}{\phi}.
\endaligned
\end{equation}

Since
\begin{equation*}\label{eqn-h-d9}
v(x, t) < 3
\end{equation*}
for all $(x, t) \in M^n\times [-T, 0]$. Thus the $w$-function
$w=v^{-1}$ has a positive lower bound on $M_t$ for all $t\in [-T,
0]$, by Propostion 6.1 in \cite{XY08} (see also section 6 in
\cite{XY09}), $M_t$ is an entire graph for any $t\in [-T, 0]$ and
$D_{R, T}(o)$ is compact. Hence $\phi f$ attains its maximum at some
point $F(p_1, t_1)$ in $D_{R, T}(o)$. By the maximum principle, at
$F(p_1, t_1)$, we have
\[
\na(\phi  f) = 0, \quad \De(\phi f) \leq 0, \quad \p_t(\phi
 f) \geq 0.
\]
Thus from (\ref{eqn-h-d8}), we obtain
\begin{equation}\label{eqn-h-d10}\aligned
 \left( k\ep_0 -3 \right)e^{-3k} \phi f^2 \leq & -  f(\De-\p_t)\phi + \frac{1}{2} \phi f  |\na h|^2 \\
 &+  f \lan \na \phi , \na h \ran + \frac{2|\na \phi|^2 f}{\phi}.
\endaligned
\end{equation}
Note that
\begin{equation}\label{eqn-h-d100}
|\na h| \leq |h'(\cdot, P_0)\circ \gamma|\cdot |d\gamma| = k |B|\leq
|B| h,
\end{equation}
where we used $k \leq kv= h (v\geq 1)$ in the last $``\leq"$. This
implies that
\[
\phi f |\na h|^2 \leq \phi f |B|^2h^2 =  \phi |B|^4 e^h h^2 = \phi
f^2 h^2 e^{-h}.
\]
Since
\[
\lim_{h\to +\infty} h^2 e^{-h} = 0.
\]
Hence for any $\ep \in (0, 3)$, there exists a positive constant
$h_0$, such that if $h>h_0$, then we have
\begin{equation}\label{eqn-h-d11}
h^2e^{-h} < \frac{\ep}{3}.
\end{equation}
Thus when $k> h_0$,
\begin{equation*}\label{eqn-h-d111}
h=kv \geq k > h_0,
\end{equation*}
then we arrive at (\ref{eqn-h-d11}). Therefore we obtain
\begin{equation}\label{eqn-h-d12}\aligned
 \frac{1}{2} \phi f  |\na h|^2 \leq\frac{1}{2} \phi f^2 h^2 e^{-h} \leq \frac{\ep}{6} \phi  f^2.
\endaligned
\end{equation}
The Schwarz inequality and (\ref{eqn-h-d100}) imply that
\[
 f \lan \na h, \na \phi \ran \leq  f |\na
h|\cdot |\na \phi| \leq  f |B|h\cdot |\na \phi| = f^{\frac{3}{2}} h
e^{-\frac{h}{2}}|\na \phi|.
\]
From (\ref{eqn-h-d11}), we know that when $k> h_0$, we get
\[
h e^{-\frac{h}{2}}<1.
\]

Let
\[
k > \max \left\{ \frac{3}{\ep_0},  h_0\right\}.
\]
Then we have
\begin{equation}\label{eqn-h-d13}\aligned
 f \lan \na h, \na \phi \ran \leq&   f^{\frac{3}{2}}|\na \phi| =  f^{\frac{3}{2}} \phi^{\frac{3}{4}}\cdot \frac{ |\na \phi|}{\phi^{\frac{3}{4}}}\\
\leq & \frac{\ep}{6} \left(  f^{\frac{3}{2}} \phi^{\frac{3}{4}} \right)^{\frac{4}{3}} + C(\ep) \left( \frac{ |\na \phi|}{\phi^{\frac{3}{4}}} \right)^{4} \\
\leq & \frac{\ep}{6} \phi  f^2 + C( \ep) \frac{1}{R^4}.
\endaligned
\end{equation}
Direct computation gives
\begin{equation*}\label{eqn-main}\aligned
(\De-\p_t) r^2 = (\De-\p_t) |F|^2 = 2n.
\endaligned
\end{equation*}
and
\[
(\De-\p_t) r^2 = 2|\na r|^2 +2r(\De-\p_t) r \geq 2r(\De-\p_t )r,
\]
The above two formulas imply that
\[
(\De-\p_t)r \leq \frac{n}{r}.
\]
It follows that
\begin{equation}\label{eqn-h-d14}\aligned
 - f(\De-\p_t)\phi = & - f\p_r^{2}\eta |\na r|^2 - f\p_r\eta (\De-\p_t) r + f \p_t\eta \\
 \leq & - f\p_r^{2}\eta |\na r|^2 - f\p_r\eta \cdot\frac{n}{r} +  f \p_t\eta \\
 \leq &  f|\p_{r}^2 \eta| +  f |\p_r \eta|\cdot \frac{2n}{R} +  f|\p_t \eta| \\
 =& f\phi^{\frac{1}{2}} \cdot \frac{|\p_r^{2}\eta|}{\phi^{\frac{1}{2}}} + f\phi^{\frac{1}{2}} \cdot\frac{|\p_r \eta| \frac{2n}{R}}{\phi^{\frac{1}{2}}} +  f\phi^{\frac{1}{2}} \cdot \frac{|\p_t\eta|}{\phi^{\frac{1}{2}}} \\
 \leq &  \frac{\ep}{6} \phi f^2 + C(\ep) \left( \frac{|\p_r^{2}\eta|}{\phi^{\frac{1}{2}}} \right)^2 +  \frac{\ep}{6} \phi f^2 +C(\ep) \left(\frac{2n}{R}\right)^2 \cdot\left(\frac{ |\p _r \eta|}{\phi^{\frac{1}{2}}} \right)^2 \\
 &+  \frac{\ep}{6} \phi f^2 +C(\ep) \left( \frac{|\p_t\eta|}{\phi^{\frac{1}{2}}} \right)^2 \\
 \leq &  \frac{\ep}{2} \phi f^2 + C(\ep) \frac{1}{R^4} + C(n, \ep) \frac{1}{R^4} + C(\ep) \frac{1}{T^2}.
\endaligned
\end{equation}
(Note that since $\p_r\eta=0$ for $r\leq \frac{R}{2}$, then we may
assume that $r\geq \frac{R}{2}$ in the second inequality).

The last term of the right hand side of (\ref{eqn-h-d10}) can be
estimated as follows
\begin{equation}\label{eqn-h-d15}\aligned
\frac{2|\na \phi|^2 f}{\phi} =  f\phi^{\frac{1}{2}}\cdot \frac{2|\na
\phi|^2}{\phi^{\frac{3}{2}}} \leq \frac{\ep}{6} \phi f^2 +C(\ep)
\left( \frac{|\na\phi|}{\phi^{\frac{3}{4}}} \right)^4 \leq
\frac{\ep}{6} \phi f^2 +C(\ep) \frac{1}{R^4}.
\endaligned
\end{equation}
Substituting (\ref{eqn-h-d12})--(\ref{eqn-h-d15}) into
(\ref{eqn-h-d10}), we have
\begin{equation}\label{eqn-h-d16}\aligned
 \left( k\ep_0 -3 \right)e^{-3k} \phi f^2 \leq & \ep \phi f^2 + C(n, \ep) \frac{1}{R^4} + C(\ep) \frac{1}{T^2}.
\endaligned
\end{equation}
Choosing $\ep>0$ sufficiently small, such that
\[
 \left( k\ep_0 -3 \right)e^{-3k} - 2\ep > 0.
\]
Then we derive
\begin{equation*}\label{eqn-h-d16}\aligned
 \phi f^2 \leq &C\left( \frac{1}{R^4} + \frac{1}{T^2} \right),
\endaligned
\end{equation*}
where $C>0$ is a constant which is independent of $R$ and $T$. Hence
we obtain
\[
\sup_{D_{\frac{R}{2}, \frac{T}{2}}(o)} |B| \leq
\sup_{D_{\frac{R}{2}, \frac{T}{2}}(o)} |B|e^{\frac{h}{2}}\leq C
\left( \frac{1}{R} + \frac{1}{\sqrt{T}} \right).
\]
\end{proof}

As a consequence, we derive a rigidity result of ancient solutions
as follows.

\begin{theorem}\label{thm-anc6}

Let $F:M^n\times (-\infty, 0] \to \mathbb{R}^{m+n}$ be a complete
ancient solution to the mean curvature flow ($m\geq 2$). If for any
constant $v_0<3$, as $t\to -\infty$, the slope function $v$
satisfies
\[
 v(x, t) \leq v_0 < 3,
\]
then $M_t$ has to be affine linear for any $t\in (-\infty, 0]$.

\end{theorem}

\begin{remark}

Guan-Xu-Zhao \cite{GXZ23} showed that the conclusion holds under the
assumption that $v\leq v_1< \sqrt{2}$ and uniformly bounded mean
curvature (see also the related work \cite{Qiu22}).
\end{remark}

In particular, the translating solitons are special ancient
solutions. Hence we have

\begin{theorem}\label{thm-anc7}

Let $M^n$ be a complete $n$-dimensional translating soliton in
$\mathbb{R}^{m+n}$ with codimension $m\geq 2$. If for any constant
$v_0<3$, the slope function  $v$ satisfies
\[
v\leq v_0 < 3,
\]
then $M^n$ is an affine subspace.

\end{theorem}

\begin{remark}

The first author \cite{Qiu23} showed that the same conclusion holds
under the assumption that $v\leq v_0 < 2$. See also the related
works by \cite{GXZ23}, \cite{Qiu22b}, \cite{Xin15}.

\end{remark}

We also study complete ancient solutions whose Gauss images lie in a compact subset of
the largest geodesic convex subset $B_{JX}(P_0)$ (see
(\ref{eqn-LGCS})), and carry out local estimates of the second
fundamental form.

\begin{theorem}\label{thm-anc8}

Let $F:M^n \times [-T, 0] \to \mathbb R^{m+n}$ be a solution to the
mean curvature flow ($m\geq 2$). Assume that the initial submanifold has bounded slope $v\le v_0<\infty$ and its image
 under the Gauss map $\gamma: M^n \to \grs{n}{m}$ is contained in $B^{\lambda_0}_{JX}(P_0):= \{ P\in \mathbb
 U: \sup_{i\neq j}\lambda_i\lambda_j \leq \lambda_0<1 \}$ of the largest
geodesic convex subset $B_{JX}(P_0)$ as in (\ref{eqn-LGCS}), where
$v_0>0$ and $\lambda_0>0$ are constants. Then there exists a
positive constant $C$ which is independent of $R$ and $T$, such that
\begin{equation}\label{es1}
\sup_{D_{\frac{R}{2}, \frac{T}{2}}(o)} |B| \leq C \left( \frac{1}{R}
+ \frac{1}{\sqrt{T}} \right).
\end{equation}

\end{theorem}

\begin{proof}

Denote $h_1:=k_1 v$, where $k_1$
is a positive constant which will be determined later.
It follows that from (\ref{ev2})
\begin{equation}\label{eqn-BJ16}\aligned
(\Delta- \partial_t)h_1 =& k_1(\Delta-\partial_t)v  \geq
k_1(1-\lambda_0)|B|^2.
\endaligned
\end{equation}
Define the function on $\mathcal{M}_T$ by
\[
\widetilde f = |B|^2 e^{h_1}.
\]
By the assumption that $v\leq v_0<\infty$, then by a similar proof
of Theorem \ref{thm5}, we can conclude that for $k_1$ large enough,
\begin{equation}\label{eqn-BJ17}\aligned
 \left( k_1(1-\lambda_0) -3 \right)e^{-k_1 v_0} \phi\widetilde f^2 \leq & \ep \phi\widetilde f^2 + C(n, \ep) \frac{1}{R^4} + C(\ep) \frac{1}{T^2}.
\endaligned
\end{equation}
Choosing $\ep>0$ sufficiently small, such that
\[
 \left( k_1(1-\lambda_0) -3 \right)e^{-k_1 v_0} - 2\ep > 0.
\]
Then we get
\begin{equation*}\label{eqn-h-d16}\aligned
 \phi\widetilde f^2 \leq &C\left( \frac{1}{R^4} + \frac{1}{T^2} \right),
\endaligned
\end{equation*}
where $C>0$ is a constant which is independent of $R$ and $T$. Hence
we derive
\[
\sup_{D_{\frac{R}{2}, \frac{T}{2}}(o)} |B| \leq
\sup_{D_{\frac{R}{2}, \frac{T}{2}}(o)} |B|e^{\frac{h_1}{2}}\leq C
\left( \frac{1}{R} + \frac{1}{\sqrt{T}} \right).
\]
\end{proof}

Consequently, we have the following rigidity result of ancient
solutions.

\begin{theorem}\label{thm-anc9}

Let $F:M^n\times (-\infty, 0] \to \mathbb{R}^{m+n}$ be a complete
ancient solution to the mean curvature flow ($m\geq 2$). If the
initial submanifold has bounded slope $v\le v_0<\infty$ and its
image of the Gauss map is contained in $B^{\lambda_0}_{JX}(P_0):= \{
P\in \mathbb U: \sup_{i\neq j}\lambda_i\lambda_j \leq \lambda_0<1
\}$, where $v_0>0$ and $\lambda_0>0$ are constants. Then $M_t$ has
to be affine linear for any $t\in (-\infty, 0]$.

\end{theorem}

An important class of ancient solutions are self-shrinkers,
satisfying a system of quasi-linear elliptic PDE of the second order
\begin{equation}\label{eqn-S111}
H= -\frac{1}{2}X^N,
\end{equation}
where $X$ is the position vector of $M^n$ in $\mathbb R^{m+n}$ and
$X^N$ denotes the orthogonal projection of $X$ onto the normal
bundle of $M^n$.

Recall that $M^n$ is said to be a translating soliton in
$\mathbb{R}^{m+n}$ if it
 satisfies
\begin{equation}\label{eqn-T111}
H= V_{0}^N,
\end{equation}
where  $V_0$ is a fixed vector in $\mathbb{R}^{m+n}$ with unit
length and $V_{0}^N$ denotes the orthogonal projection of $V_0$ onto
the normal bundle of $M^n$.

Self-shrinkers and translating solitons are both significant
examples of ancient solutions to the mean curvature flow.

In \cite{DJX21}, the authors found a 2-bounded subset of $\mathbb U$
defined by
\begin{equation}\label{eqn-T1}
\mathbb T^{2, \Lambda}:= \{ P\in \mathbb U: \lambda_i\lambda_j <
\Lambda \quad {\rm for} \quad {\rm every}\quad i\neq j \}.
\end{equation}
It is easy to see that
\[
\mathbb T^{2,1}= B_{JX}(P_0).
\]
Since $B_{JX}(P_0)$ is the largest geodesic convex subset, therefore
the distance function from $P_0$ is convex on $T^{2,1}$, but it is
no longer convex on $T^{2,\Lambda}$ when $\Lambda>1$.

\begin{theorem}\label{thm-anc10}

Let $M^n$ be a complete $n$-dimensional self-shrinker or translating
soliton in $\mathbb{R}^{m+n}$ with codimension $m\geq 2$ and bounded
slope $v\le v_0<\infty$, where $v_0$ is a positive constant. If the
image of $M^n$ under the Gauss map $\gamma: M^n \to \grs{n}{m}$ is
contained in $\mathbb T^{2, \sqrt 2}_{\Lambda}:= \{ P\in \mathbb U:
\sup_{i\neq j}\lambda_i\lambda_j \leq \Lambda<\sqrt 2 \}$, where
$\Lambda>0$ is a constant. Then $M^n$ is an affine subspace.

\end{theorem}

\begin{proof}
Note that $\sup_{i\neq j}\lambda_i\lambda_j \leq \Lambda < \sqrt 2$.
Then by the formulas (7.3), (7.13) in \cite{DJX21} and
(\ref{eqn-BJ5}), we have
\begin{equation}\label{eqn-T2}\aligned
\sum_{i}{\rm Hess}(\log v(\cdot, P_0))(d\gamma(e_i),
d\gamma(e_i))\circ \gamma \geq \eps |B|^2,
\endaligned
\end{equation}
where $\eps>0$ is a constant which depends on $\Lambda$.

When $M^n$ is a self-shrinker, let $\mathcal{L}:= \De -
\frac{1}{2}\left< X, \na \cdot\right>$. Then by
 Corollary 2.1 in \cite{DXY} and (\ref{eqn-T2}), we get
\begin{equation}\label{eqn-T3}\aligned
\mathcal{L} v = & \sum_{i}{\rm Hess}(v(\cdot, P_0))(d\gamma(e_i),
d\gamma(e_i))\circ \gamma \\
\geq & v \sum_{i}{\rm Hess}(\log v(\cdot, P_0))(d\gamma(e_i),
d\gamma(e_i))\circ \gamma \\
\geq & \eps |B|^2.
\endaligned
\end{equation}
From the formula (2.7) in \cite{DX}, we have
\begin{equation}\label{eqn-TB1}\aligned
\mathcal{L} |B|^2\geq 2|\na|B||^2+|B|^2-3|B|^4\geq
2|\na|B||^2-3|B|^4.
\endaligned
\end{equation}

When $M^n$ is a translating soliton, set  $\mathcal{L}_{II}:= \De +
\left< V_0, \na \cdot\right>$.
 Corollary 6.2 in \cite{Xin15} and the inequality (\ref{eqn-T2}) imply that
\begin{equation}\label{eqn-T33}\aligned
\mathcal{L}_{II} v = & \sum_{i}{\rm Hess}(v(\cdot,
P_0))(d\gamma(e_i),
d\gamma(e_i))\circ \gamma \\
\geq & v \sum_{i}{\rm Hess}(\log v(\cdot, P_0))(d\gamma(e_i),
d\gamma(e_i))\circ \gamma \\
\geq & \eps |B|^2.
\endaligned
\end{equation}
By Proposition 2.1 in \cite{Xin15}, we derive
\begin{equation}\label{eqn-TB2}\aligned
\mathcal{L}_{II} |B|^2 \geq 2|\na|B||^2-3|B|^4.
\endaligned
\end{equation}
From (\ref{eqn-T3}) and (\ref{eqn-T33}) we see that the estimates of
$\mathcal{L} v$ on the self-shrinker and $\mathcal{L}_{II} v$ on the
translating soliton are the same, so are the estimates of
$\mathcal{L} |B|^2$ and $\mathcal{L}_{II} |B|^2$. Hence, in the
following, we only prove the conclusion for translating solitons.

 Denote $h_2:= k_2 v$, where $k_2$ is a positive constant which
will be determined later. From (\ref{eqn-T3}), we get
\begin{equation}\label{eqn-T4}\aligned
\mathcal{L}_{II} h_2 \geq k_2\eps |B|^2.
\endaligned
\end{equation}
Let
\[
f_2:= |B|^2 e^{h_2}.
\]
Then by a similar proof to get (\ref{eqn-h-d5}), we can conclude
that
\begin{equation}\label{eqn-T5}\aligned
\mathcal{L}_{II} f_2 \geq \left(k_2\eps -3\right)|B|^4 e^{h_2} -
\frac{1}{2}f_2 |\na h_2|^2 +\left< \na f_2, \na h_2 \right>.
\endaligned
\end{equation}
By the assumption that $v\leq v_0<\infty$, thus
\[
e^{h_2} \geq e^{-k_2 v_0}e^{2h_2}.
\]
Then from (\ref{eqn-T5}), we obtain
\begin{equation}\label{eqn-T6}\aligned
\mathcal{L}_{II} f_2 \geq \left(k_2\eps-3\right) e^{-k_2 v_0}f_2^{2}
- \frac{1}{2}f_2 |\na h_2|^2 +\left< \na f_2, \na h_2 \right>.
\endaligned
\end{equation}
Denote by $X$ the position vector of the translator and set $r=|X|$.
Then,
\begin{equation}\label{eqn-dist}\aligned
\na r^2 =& 2X^T, \quad \quad |\na r| \leq 1 \\
\mathcal{L}_{II} r^2 =& 2n + 2\left< H, X \right> + \left< V_0^{T}, 2X^T \right> \\
= & 2n + 2\left< V_0^{N}, X \right> + 2\left< V_0^{T}, X \right> \\
=& 2n + 2\left< V_0, X\right> \leq 2n+2r.
\endaligned
\end{equation}
Since
\begin{equation}\label{eqn-dist2}\aligned
\mathcal{L}_{II} r^2 = 2r\mathcal{L}_{II} r +2|\na r|^2 \geq
2r\mathcal{L}_{II} r.
\endaligned
\end{equation}
Therefore from (\ref{eqn-dist}) and (\ref{eqn-dist2}), we get
\begin{equation}\label{eqn-dist3}\aligned
\mathcal{L}_{II} r\leq \frac{n+r}{r}.
\endaligned
\end{equation}

 Let $B_R(o)\subset \mathbb{R}^{m+n}$ be a closed ball of radius
$R$ centered at the origin $o\in \mathbb{R}^{m+n}$ and
$D_{R}(o)=M\cap B_R(o)$. Since the slope function $v\leq
v_0<\infty$, then the $w$-function $w$ has a positive lower bound.
Hence by Proposition 6.1 in \cite{XY08}, $M^n$ is an entire graph
and $D_R(o)$ is compact.

 Let $\widetilde \eta(t)\in C^2[0, +\infty)$ such that
\begin{equation*}
\widetilde \eta(t)=\begin{cases}\aligned
1,\quad &  t\in [0, \frac{1}{2}],\\
0,\quad &  t\in [1, +\infty),
\endaligned
\end{cases}
\end{equation*}
\[
\widetilde \eta(t) \in [0, 1], \quad \widetilde \eta'(t) \leq
0,\quad \eta''(t)\geq -C_0 \quad {\rm and} \quad
\frac{|\widetilde\eta'(t)|^2}{\widetilde \eta(t)}\leq C_0,
\]
where $C_0$ is a positive constant.

Let $\eta_1(x)= \widetilde \eta(\frac{r(x)}{R})$, then the estimates
(\ref{eqn-dist}) and (\ref{eqn-dist3}) imply that
\begin{equation}\label{eqn-eta1}\aligned
\frac{|\na \eta_1|^2}{\eta_1} = \frac{|\widetilde \eta'|^2|\na
r|^2}{R^2\widetilde \eta} \leq \frac{C_0}{R^2},
\endaligned
\end{equation}
\begin{equation}\label{eqn-eta2}\aligned
\mathcal{L}_{II} \eta_1 = \frac{\widetilde \eta' \mathcal{L}_{II}
r}{R} + \frac{\widetilde \eta''|\na r|^2}{R^2}\geq -\sqrt C_0\cdot
\frac{2n+R}{R^2} - \frac{C_0}{R^2}.
\endaligned
\end{equation}
From (\ref{eqn-T6}), we have
\begin{equation}\label{eqn-T7}\aligned
\mathcal{L}_{II} (f_2\eta_1) = & (\mathcal{L}_{II} f_2)\eta_1 +
2\left< \na f_2, \na \eta_1
\right> + f_2\mathcal{L}_{II} \eta_1 \\
 \geq & \left(k_2\eps-3\right) e^{-k_2 v_0}f_2^{2}\eta_1
-
\frac{1}{2}f_2\eta_1 |\na h_2|^2 \\
&+\left< \na f_2, \na h_2 \right>\eta_1 +2\eta_1^{-1}\left<
\na(f_2\eta_1),
\na\eta_1 \right> \\
&-2\eta_1^{-1}f_2|\na\e_1|^2 + f_2\mathcal{L}_{II} \eta_1.
\endaligned
\end{equation}
By (\ref{eqn-eta1})--(\ref{eqn-T7}), we derive
\begin{equation}\label{eqn-T77}\aligned
\mathcal{L}_{II} (f_2\eta_1)
 \geq & \left(k_2\eps-3\right) e^{-k_2 v_0}f_2^{2}\eta_1
-
\frac{1}{2}f_2\eta_1 |\na h_2|^2 \\
&+\left< \na f_2, \na h_2 \right>\eta_1 +2\eta_1^{-1}\left<
\na(f_2\eta_1),
\na\eta_1 \right> \\
&-\frac{2C_0}{R^2}f_2 -\left( \sqrt C_0\cdot \frac{2n+R}{R^2} +
\frac{C_0}{R^2} \right)f_2.
\endaligned
\end{equation}
 The Cauchy inequality and (\ref{eqn-eta1}) imply
\begin{equation}\label{eqn-T8}\aligned
|f_2\left< \na\eta_1, \na h_2 \right>| \leq & f_2|\na\eta_1| |\na
h_2|
\leq f_2|\na h_2|\cdot \frac{\sqrt C_0}{R}\eta_1^{\frac{1}{2}} \\
\leq & \frac{1}{2}f_2 |\na h_2|^2 \eta_1 + \frac{C_0}{2R^2}f_2.
\endaligned
\end{equation}
From (\ref{eqn-T8}), we then obtain
\begin{equation}\label{eqn-T9}\aligned
\left< \na f_2, \na h_2 \right>\eta_1 =& \left< \na (f_2\eta_1), \na
h_2
\right> -f_2 \left< \na \eta_1, \na h_2 \right> \\
\geq & \left< \na (f_2\eta_1), \na h_2 \right> - \frac{1}{2}f_2 |\na
h_2|^2 \eta_1 - \frac{C_0}{2R^2}f_2.
\endaligned
\end{equation}
Combining (\ref{eqn-T77}) with (\ref{eqn-T9}), it follows
\begin{equation}\label{eqn-T10}\aligned
\mathcal{L}_{II} (f_2\eta_1) \geq & \left(k_2\eps-3\right) e^{-k_2 v_0}f_2^{2}\eta_1 -f_2\eta_1 |\na h_2|^2 \\
&+\left< \na (f_2\eta_1), \na h_2 + 2\eta_1^{-1}\na \eta_1 \right>\\
 &-\frac{7C_0}{2R^2}f_2 - \frac{\sqrt C_0(2n+R)}{R^2}f_2.
\endaligned
\end{equation}
By (\ref{eqn-h-d100}) and (\ref{eqn-h-d11}), we know that for any
$\ep > 0$, there exists a positive constant $\widetilde h_0$, when
$k_2> \widetilde h_0$, we have
\begin{equation}\label{eqn-T11}\aligned
f_2\eta_1 |\na h_2|^2 \leq f_2\eta_1 |B|^2 h_2^{2} = |B|^4
e^{h_2}\eta_1 h_2^{2} = f_2^{2}\eta_1 h_2^{2}e^{-h_2} < \ep
f_2^{2}\eta_1.
\endaligned
\end{equation}
Let
\[
k_2 > \max \left\{ \frac{3}{\eps}, \widetilde h_0 \right\}.
\]
 Denote $\widetilde f_2 := f_2\eta_1$. From (\ref{eqn-T10}) and
(\ref{eqn-T11}), we get
\begin{equation}\label{eqn-T12}\aligned
\mathcal{L}_{II} \widetilde f_2 \geq & \left[\left(k_2\eps-3 \right) e^{-k_2 v_0}- \ep \right] \eta_1^{-1}\widetilde f_2^{2}  \\
&+\left< \na \widetilde f_2, \na h_2 + 2\eta_1^{-1}\na \eta_1 \right>\\
 &-\frac{7C_0}{2R^2}\eta_1^{-1}\widetilde f_2 - \frac{\sqrt C_0(2n+R)}{R^2}\eta_1^{-1}\widetilde f_2.
\endaligned
\end{equation}

Since $\left.\widetilde f_2 \right|_{\p_{D_R(o)}}=0$, $ \widetilde
f_2$ achieves an absolute maximum in the interior of $D_R(o)$, say
$\widetilde f_2 \leq \widetilde f_2(q)$ for some $q$ inside
$D_R(o)$.
 We may assume $|B|(q)\neq 0$. By the maximum principle, we have
\[
\na \widetilde f_2 (q) = 0, \quad\quad \mathcal{L}_{II} \widetilde
f_2(q) \leq 0.
\]
Then by (\ref{eqn-T12}), we obtain the following at $q$:
 \begin{equation}\label{eqn-T13}\aligned
0\geq \mathcal{L}_{II} \widetilde f_2 \geq &  \left[\left(k_2\eps-3 \right) e^{-k_2 v_0}- \ep \right] \eta_1^{-1}\widetilde f_2^{2}  \\
 &-\frac{7C_0}{2R^2}\eta_1^{-1}\widetilde f_2 - \frac{\sqrt C_0(2n+R)}{R^2}\eta_1^{-1}\widetilde f_2.
\endaligned
\end{equation}
It follows that
\begin{equation}\label{eqn-T14}\aligned
\left[\left(k_2\eps-3 \right) e^{-k_2 v_0}- \ep \right]\widetilde
f_2(q) \leq \frac{7C_0}{2R^2} + \frac{\sqrt C_0(2n+R)}{R^2}.
\endaligned
\end{equation}
Choosing $\ep$ sufficiently small, such that
\[
\left(k_2\eps-3 \right) e^{-k_2 v_0}- 2\ep >0.
\]
Then we obtain
\begin{equation}\label{eqn-T14}\aligned
\widetilde f_2(q) \leq C\left(\frac{7C_0}{2R^2} + \frac{\sqrt
C_0(2n+R)}{R^2}\right),
\endaligned
\end{equation}
where $C$ is a positive constant which is independent of $R$.

For any $x \in M^n$, we can choose a sufficiently large $R$, such
that $x\in D_{\frac{R}{2}}(o)$. Thus from (\ref{eqn-T14}), we get
 \begin{equation}\label{eqn-T15}\aligned
|B|^2(x) \leq  \sup_{D_{\frac{R}{2}}(o)}\widetilde f_2 \leq
\widetilde f_2(q)  \leq  C\left(\frac{1}{R} + \frac{1}{R^2}\right).
\endaligned
\end{equation}
Letting $R \to +\infty$ in (\ref{eqn-T15}), then we derive $B\equiv
0$. Hence $M^n$ is an affine subspace.
\end{proof}

\begin{remark}

By Jost-Xin's work \cite{JX}, $\mathbb U_2:= \{ P \in \mathbb U:
v(P, P_0) < 2 \} \subset B_{JX}(P_0)=\mathbb T^{2,1}\subset \mathbb
T^{2,\sqrt 2}_{\Lambda}$ and the open geodesic ball of $\grs{n}{m}$
centered at $P_0$ of radius $\frac{\sqrt 2 \pi}{4}$ is contained in
$B_{JX}(P_0)\subset \mathbb T^{2, \sqrt 2}_{\Lambda}$. Thus Theorem
\ref{thm-anc10} improves Corollary 1 and Corollary 2 in
\cite{Qiu23}.


\end{remark}

\vskip24pt

\bibliographystyle{amsplain}

\end{document}